%% file: main.tex
\documentclass[12pt,reqno]{amsart}
\usepackage{amsfonts, amsmath, amssymb, amsthm, mathtools} 
\usepackage{fullpage}
\usepackage[english]{babel}
\usepackage{amsopn}
\usepackage{textcomp, cmap, comment}	  
\usepackage{graphicx, wrapfig}                           
\usepackage{epigraph, color, hyperref}                   
\usepackage{array,tabularx,tabulary,booktabs}             
\usepackage{euscript, mathrsfs}	   

\usepackage{pgf,tikz}
\usepackage{mathrsfs}
\usetikzlibrary{arrows}

\DeclareMathOperator{\p}{p}
\DeclareMathOperator{\pr}{pr}
\DeclareMathOperator{\dist}{dist}
\DeclareMathOperator{\conv}{conv}
\DeclareMathOperator{\aff}{aff}
\DeclareMathOperator{\vol}{vol}

\newcounter{iii}                                            

\theoremstyle{plain}
\newtheorem{theo}{Theorem}

\newtheorem{lem}{Lemma}

\theoremstyle{definition}

\author{A.~Polyanskii}
\title{Pairwise intersecting homothets of a convex body
}
\address{Moscow Institute of Physics and Technology, Technion, Institute for Information Transmission Problems RAS.}
\email{alexander.polyanskii@yandex.ru}
\date{\today}

\begin{document}
\begin{abstract}
We show that the maximum number of pairwise intersecting positive homothets of a $d$-dimensional centrally symmetric convex body, none of which contains the center of another in its interior, is at most $3^{d+1}$. Also, we improve upper bounds for cardinalities of  $k$-distance sets in Minkowski spaces.
\end{abstract}
\maketitle

\section{Introduction}

A \emph{convex body} $K$ in the $d$-dimensional Euclidean space $\mathbb R^d$ is a compact convex set with non-empty interior, and it is \emph{$o$-symmetric} if $K = -K$. A \emph{homothet} of $K$ is a set of the form $\mathbf v+\lambda K := \{\mathbf v+ \lambda \mathbf k : \mathbf k \in K \}$, where $\lambda\in\mathbb R$ is the homothety ratio, and $\mathbf v\in \mathbb R^d$ is a translation vector. A homothet of $K$ is called \emph{positive} if its homothety ratio is positive. We will consider only positive homothets of $o$-symmetric bodies here, and thus we will omit the word "positive" most of the time. Also, we write $[n]$ for the set $\{1,2,\dots, n\}$, $\dist(h_1,h_2)$ for the Euclidean distance between two parallel hyperplanes $h_1$ and $h_2$, $\dim(h)$ for the dimension of a flat $h$. $\conv(A)$, $\aff(A)$, $\vol(A)$ and $\partial A$ stand for the convex hull, the affine hull, the volume and the boundary of a set $A\subset \mathbb R^d$ respectively.

\emph{A Minkowski arrangement} of an $o$-symmetric convex body $K$ is called a family $\{\mathbf v_i + \lambda_i K\}$ of positive homothets of $K$ such that none of the homothets contains the center of any other homothet in its interior (see \cite{FT67}). We write $\kappa(K)$ for the largest number of homothets that a pairwise intersecting Minkowski arrangement of $K$ can have. Z.~F\"uredi and P.A.~Loeb \cite{FL94} proved that $\kappa(K)\leq 5^d$. Recently, M.~Nasz\'odi, J.~Pach and K.~Swanepoel~\cite{NPS16} improved this result to $\kappa(K)\leq O(3^d d\log d)$. The authors of~\cite{NPS16} noted that it is obvious that for the $d$-dimensional cube $C^d$ we have $\kappa(C^d)=3^d$. We prove the following upper bound for $\kappa(K)$, which is sharp up to the constant factor.

\begin{theo} \label{thm:symmetric}
For any $d$-dimensional $o$-symmetric convex body $K$,
\[
\kappa(K)\leq 3^{d+1}.
\]
\end{theo}

Also, some generalization of a Minkowski arrangement for non-symmetric bodies (the role of the center is
played by an arbitrary interior point) was studied in \cite{NPS16}. Unfortunately, it is impossible to generalize our approach for non-symmetric bodies.

We call a subset $S$ of a metric space a \emph{$k$-distance set} if the set of non-zero distances occurring between points of $S$ is of size at most $k$. A $1$-distance set is called an \emph{equilateral set}. For $d$-dimensional Minkowski spaces it is well known that the maximal cardinality of an equilateral (that is, a $1$-distance) set is $2^d$ with equality  iff the unit ball of the space is a parallelotope, see~\cite{P71}. K.~Swanepoel~\cite{S07} proved that if the unit ball of a $d$-dimensional Minkowski space is a parallelotope then a $k$-distance set has cardinality at most $(k+1)^{d}$, where the bound is tight. Therefore, he \cite{S07} conjectured that the maximal cardinality of $k$-distance sets in Minkowski spaces is $(k+1)^d$. Also, it was proved in~\cite{S07} that the cardinality of a $k$-distance set in a $d$-dimensional Minkowski space is at most $\min\{2^{kd}, (k+1)^{(11^d-9^d)/2}\}$. Moreover, the last bound was recently replaced by $(k+1)^{5^{d+o(d)}}$, see~\cite{S16}. Our second result is the following improvement.
\begin{theo}\label{thm:kdistancessets}
The cardinality of a $k$-distance set ($k>1$) in a $d$-dimensional Minkowski space is at most $k^{O(3^dd)}$, where the constant in $O(\cdot)$ does not depend on $d$ and $k$.
\end{theo}

Our proof is based on Theorem~\ref{thm:basic}, which seems to be of independent interest.

\begin{theo}\label{thm:basic}
Assume that $\mathbf v_1, \mathbf v_2, \dots, \mathbf v_n$ are points in a $d$-dimensional Minkowski space with an $o$-symmetric convex body $K$ as the unit ball, such that $\|\mathbf v_i-\mathbf v_j\|_K=\lambda_i$ for any $1\leqslant i<j\leqslant n$, where $\lambda_i$, $i\in[n-1]$, are some positive numbers. Then 
\[
n\leq d\left(1+\frac{2}{2-2^{1/(d-1)}}\right)^{d+1}=O(3^d d).
\]
\end{theo} 
It is important to note that M.~Nasz\'odi, J.~Pach and K.~Swanepoel \cite{NPS16} proved that if the conditions of Theorem~\ref{thm:basic} hold then $n= O(6^d (d \log d)^2)$.

For more links dealing with $k$-distance sets we refer the interested readers to \cite{S07, S16}.

One of the main ingredients of the proofs of Theorems \ref{thm:symmetric} and \ref{thm:basic} is the following simple lemma which is a generalization of the well-known Danzer-Gr\"unbaum Theorem about the maximal cardinality antipodal sets, i.e. such sets that satisfy conditions of Lemma~\ref{lem:general} when $\lambda=1$ (see~\cite{DG62} and also Lemma~7 in~\cite{NPS16}). 
\begin{lem}\label{lem:general}
Suppose that $\lambda\geq 1$ is a real number and $X=\{\mathbf x_1,\dots, \mathbf x_n\}\subset \mathbb R^d$ is a set of points such that for any $i\ne j\in [n]$ there are two distinct parallel hyperplanes $k_{i,j}$ and $k_{j,i}$ with $X\subset \conv(k_{i,j}, k_{j,i})$  and
\begin{equation}\label{eq:ratio}\frac{\dist(k_{i,j}, k_{j,i})}{\dist(g_{i,j}, g_{j,i})}\leq \lambda,\end{equation} where $g_{i,j}$ and $g_{j,i}$ are hyperplanes passing through $\mathbf x_i$ and $\mathbf x_j$ respectively and parallel to $k_{i,j}$ (and $k_{j,i}$). Then $n\leq (1+\lambda)^d$.
\end{lem}

Another key tool in our proofs is the lifting method developed in~\cite{N06} (see also~\cite{LN09}), where M.~Nasz\'odi showed that the maximal number of pairwise touching positive homothets of a convex body $K$ that is not necessary $o$-symmetric is at most $2^{d+1}$. We develop this method further by new ideas.

The article is organized in the following way. In Section~\ref{sec:lemgen} we prove Lemma~\ref{lem:general}. In Section~\ref{sec:auxilirylemmas} we discuss some properties of a set of pairwise intersecting homothets, which we will use in Sections~\ref{sec:proofthm1} and \ref{sec:proofthm2}, where we present the proofs of Theorem~\ref{thm:symmetric} and Theorem~\ref{thm:basic} respectively. In Section~\ref{sec:proofthm3} we prove Theorem~\ref{thm:kdistancessets} using Theorem~\ref{thm:basic}.

\section{Auxiliary Lemmas}

\subsection{Proof of Lemma~\ref{lem:general}}\label{sec:lemgen}
We may clearly assume that $P:=\conv(X)$ is a $d$-dimensional polytope in $\mathbb R^d$, otherwise $d_1:=\dim(\aff(P))<d$, i.e. by induction hypothesis, we have $n\leq (1+\lambda)^{d_1}<(1+\lambda)^d$. It is easy to see that $P_i=\mathbf x_i +\frac{1}{1+\lambda}(P-\mathbf x_i)\subset P$. Without loss of generality we assume that $\mathbf x_i$ is closer to $k_{i,j}$ than $\mathbf x_j$. We claim that $P_i$ and $P_j$ do not share a common interior point. Indeed, $P_i\subset \conv(k_{i,j}\cup l_{i,j})$, $P_j\subset \conv(k_{j,i}\cup l_{j,i})$, where $l_{i,j}=\mathbf x_i+\frac{1}{1+\lambda}(k_{j,i}-\mathbf x_i)$, $l_{j,i}=\mathbf x_j+\frac{1}{1+\lambda}(k_{i,j}-\mathbf x_j)$. Note that $\conv(k_{i,j}\cup l_{i,j})$ and $\conv(k_{j,i}\cup l_{j,i})$ do not have a common interior point because
\begin{gather*}\dist(k_{i,j},l_{i,j})+dist (k_{j,i},l_{j,i})=\\
=\dist(k_{i,j},g_{i,j})+\frac{1}{1+\lambda}\dist(g_{i,j},k_{j,i})+\dist(k_{j,i}, g_{j,i})+\frac{1}{1+\lambda}\dist(g_{j,i},k_{i,j})=\\
=\dist(k_{i,j},k_{j,i})-\dist(g_{i,j},g_{j,i})+\frac{1}{1+\lambda} \dist(k_{i,j},k_{j,i})+\frac{1}{1+\lambda}\dist(g_{i,j},g_{j,i}) \leq \dist(k_{i,j},k_{j,i}).
\end{gather*}
The last inequality holds because of \eqref{eq:ratio}.
Therefore, $\sum_{i=1}^n \vol(P_i)\leq \vol(P)$, i.e. $\frac{n}{(1+\lambda)^d}\vol(P)\leq \vol(P)$, $n\leq (1+\lambda)^d$. Lemma \ref{lem:general} is proved.
\subsection{Properties of pairwise intersecting homothets} \label{sec:auxilirylemmas}

Throughout Section \ref{sec:auxilirylemmas}, $\ell(\mathbf x, \mathbf y)$ denotes the line passing through points $\mathbf x$ and $\mathbf y$, $\ell(\mathbf x, l)$ and $h(\mathbf x, h)$ stand for the line and the $k$-dimensional flat passing through a point $\mathbf x$ and parallel to a line $l$ and to a $k$-dimensional flat $h$ respectively, we write $[\mathbf x, \mathbf y]$ for the segment with endpoints $\mathbf x$ and $\mathbf y$, $\Delta(\mathbf x, \mathbf y, \mathbf z)$ denotes the triangle with vertices $\mathbf x, \mathbf y, \mathbf z$. We write $\Delta(\mathbf x_1, \mathbf y_1, \mathbf z_1)\sim \Delta(\mathbf x_2, \mathbf y_2, \mathbf z_2)$ if the triangles $\Delta(\mathbf x_1, \mathbf y_1, \mathbf z_1)$ and $\Delta(\mathbf x_2, \mathbf y_2, \mathbf z_2)$ are similar. $(\mathbf x_1,\mathbf x_2;\mathbf x_3, \mathbf x_4)$ stands for the \emph{cross-ratio} of points $\mathbf x_1, \mathbf x_2,\mathbf x_3, \mathbf x_4$ on the real line, i.e. 
\[
(\mathbf x_1, \mathbf x_2; \mathbf x_3,\mathbf x_4)=\frac{x_1-x_3}{x_2-x_3}:\frac{x_1-x_4}{x_2-x_4},
\]
where $x_1, x_2, x_3, x_4$ are coordinates of the points $\mathbf x_1, \mathbf x_2, \mathbf x_3, \mathbf x_4$ respectively. If one of the points is the point at infinity then the two distances involving that point are dropped from the formula. Also, we will use the fact that if $\p:\mathbb R^d\to \mathbb R^d$ is a projective transformation and distinct points $\mathbf x_1, \mathbf x_2, \mathbf x_3, \mathbf x_4\in \mathbb R^d$ are collinear then $\p(\mathbf x_1), \p(\mathbf x_2), \p(\mathbf x_3), \p(\mathbf x_4)$ are also collinear and 
\[
(\p(\mathbf x_1), \p(\mathbf x_2); \p(\mathbf x_3), \p(\mathbf x_4))= (\mathbf x_1, \mathbf x_2;\mathbf x_3, \mathbf x_4).
\]

Let us identify $\mathbb R^d$ with the $d$-dimensional flat 
\[
h:=\{(x_1, \ldots, x_{d+2}) \in \mathbb R^{d+2} : x_{d+1} = 0, x_{d+2}=1\}\text{ in }\mathbb R^{d+2}
\]
and consider the following hyperplanes
\begin{gather*}h_0:=\{(x_1,  \ldots,x_{d+2}) \in \mathbb R^{d+2} : x_{d+2}=1\}\text{ in }\mathbb R^{d+2},\\
h_1:= \{(x_1, \ldots ,x_{d+2}) \in \mathbb R^{d+2} : x_{d+1} = 1\}\text{ in }\mathbb R^{d+2}.\end{gather*}
Note that $h\subset h_0$. Let $\{\mathbf e_i: i\in [d+2]\}$ be the standard basis of $\mathbb R^{d+2}$.

Let $K$ be an $o$-symmetric $d$-dimensional convex body such that 
\[K\subset h':=\{(x_1,\ldots, x_{d+2})\in \mathbb R^{d+2}:x_{d+1}=0, x_{d+2}=0\}.
\]
Note that the $d$-dimensional flat $h'$ is parallel to the $d$-dimensional flat $h$. Suppose that $\{\mathbf v_i:i\in [n]\}\subset h$ is a set of $n$ distinct vectors, $\{\lambda_i:i\in [n]\}\subset \mathbb R^+$ is a set of $n$ positive scalars and $\{\mathbf v_i+\lambda_i K:i\in [n]\}\subset h$ is a finite family of pairwise intersecting positive homothets of $K$.

Section~\ref{sec:auxilirylemmas} is organized in the following way. First, we define the set $X_0:=\{\mathbf x_i:=\mathbf v_i+\lambda_i \mathbf e_{d+1}: i\in[n]\}\subset h_0$ of $n$ points and prove some properties of $X_0$. Second, we apply on $X_0$ the central projection $\pr:h_0\to h_1$ from the origin of $\mathbb R^{d+2}$ onto the hyperplane $h_1$. Finally, we check that the image $X_1:=\{\mathbf y_i:=\pr(\mathbf x_i): i\in [n]\}\subset h_1$ of $X_0$ satisfies some properties.

\input{pic1.tex}

Choose $i\ne j\in[n]$. Write $r:=r_{i,j}:=\ell(\mathbf v_i,\mathbf v_j)\subset h$ and let $\mathbf r:=\mathbf r_{i,j}$ and $-\mathbf r$ be the points of intersection of $\partial K$ and $\ell(\mathbf o, r)$, here we assume that the vectors $\mathbf r$ and $\mathbf v_j-\mathbf v_i$ have the same direction (see Figure~\ref{figure:1}). Denote by $f:=f_{i,j}$ a supporting hyperplane of $\mathbf v_i+ \lambda_i K$ in $h$ passing through $\mathbf v_i+\lambda_i \mathbf r$, i.e. $f$ is a $(d-1)$-dimensional flat.

Let $\mathbf v'_k:=\mathbf v'_{k,i,j}$, $\mathbf x'_k:=\mathbf x'_{k,i,j}$ and $t_k:=t_{k,i,j}:=[\mathbf v'_k-\lambda_k \mathbf r, \mathbf v'_k+\lambda_k\mathbf r]\subset r$ be the projections of $\mathbf v_k,$ $\mathbf x_k$ and $\mathbf v_k+\lambda_k K$ in the direction of $f$ onto the two-dimensional plane $\pi$, where $\pi:=\pi_{i,j}$ passes through $\mathbf v_i$, $\mathbf v_j$, $\mathbf x_i$ and $\mathbf x_j$ (see Figure~\ref{figure:1}). It follows immediately that $\mathbf v_i=\mathbf v_i'$, $\mathbf v_j=\mathbf v_j'$, $\mathbf x_i=\mathbf x_i'$, $\mathbf x_j=\mathbf x_j'$, $\mathbf x'_k=\mathbf v'_k+\lambda_k\mathbf e_{d+1}$, $t_i=[\mathbf v_i -\lambda_i \mathbf r, \mathbf v_i+\lambda_i\mathbf r]$ and $t_j=[\mathbf v_j-\lambda_j \mathbf r, \mathbf v_j+\lambda_j\mathbf r]$.

We claim that the segments $t_k$ share a common point, which we will denote as $\mathbf x:=\mathbf x_{i,j}$. Indeed, any two segments $t_p$ and $t_q$ share a common point otherwise $\{\mathbf v_p +\lambda_p K\}$ and $\{\mathbf v_q+ \lambda_q K\}$ do not intersect each other. Therefore, by Helly's theorem for $\mathbb R$, we get that $t_k$ have a common point $\mathbf x$.

Let $u_i:=u_{i,j}$ and $u_j:=u_{j,i}$ be the real numbers such that 
\begin{equation}
\label{equation:uiuj}\mathbf x-\mathbf v_i= u_i\mathbf r\text{ and }\mathbf v_j-\mathbf x=u_j\mathbf r.
\end{equation}
Set (see Figure~\ref{figure:1})
\begin{gather*}
a_i:=a_{i,j}:=\ell(\mathbf v_i+\lambda_i \mathbf r,\mathbf x_i),\ 
a_j:=a_{j,i}:=\ell(\mathbf v_j-\lambda_j \mathbf r, \mathbf x_j),\\ 
b_{i}:=b_{i,j}:=\ell(\mathbf x, a_i),\ 
b_{j}:=b_{j,i}:=\ell(\mathbf x, a_j), \\
f_{0}:=f_{0,i,j}:=h(\mathbf x, f),\ 
B_{i}:=B_{i,j}:=\aff(b_i \cup f_{0}),\  
B_{j}:=B_{j,i}:=\aff(b_j \cup f_{0}).
\end{gather*}

Note that the set $X_0$ lies in the wedge formed by $B_i$ and $B_j$ in $h_0$ that lies in the halfspace $\{(x_1,\dots, x_{d+2})\in \mathbb R^{d+2}: x_{d+1}\geq 0\}$. Indeed, points $\mathbf x'_k$ lie in the angle formed by $b_i$ and $b_j$ that lies in the halfspace $\{(x_1,\dots, x_{d+2})\in \mathbb R^{d+2}: x_{d+1}\geq 0\}$ (see Figure~\ref{figure:1}). Since $\mathbf x'_k$ are the projections of $\mathbf x_k$ in the direction of $f$ onto the plane $\pi$, the points $\mathbf x_k$ lie in the corresponding wedge formed by $B_i$ and $B_j$.

Next, we apply the central projection $\pr:h_0\to h_1$ from the origin of $\mathbb R^{d+2}$ onto the hyperplane $h_1$. The image of $h$ is the "hyperplane at infinity" in $h_1$. Therefore, we proved the following lemma.

\begin{lem}\label{lem:slab} 
$k_{i,j}:=\pr(B_i)$ and $k_{j,i}:=\pr(B_j)$ are parallel hyperplanes in $h_1$ and $X_1=\pr(X_0)$ lies in the slab $\conv(k_{i,j}\cup k_{j,i})$.
\end{lem}

\input{pic2.tex}

Denote by $\mathbf z_i:=\mathbf z_{i,j}$ and $\mathbf z_j:=\mathbf z_{j,i}$ the points of intersection of $r_0:=r_{0,i,j}=\ell(\mathbf x_i, \mathbf x_j)$ with $b_i$ (or $B_i$) and $b_i$ (or $B_j$) respectively (see Figure~\ref{figure:2}). Recall that $\mathbf y_k=\pr(\mathbf x_k)$. Let $\mathbf s_i:=\mathbf s_{i,j}=\pr(\mathbf z_i)$, $\mathbf s_j:=\mathbf s_{j,i}=\pr(\mathbf z_j)$. Of course, $\mathbf s_i$ and $\mathbf s_j$ are the points of intersection of $\ell(\mathbf y_i, \mathbf y_j)$ with $k_{i,j}$ and $k_{j,i}$ respectively because central projections preserve lines. Denote by $g_{i,j}$ and $g_{j,i}$ the hyperplanes in $h_1$ that are parallel to $k_{i,j}$ and $k_{j,i}$ and pass through $\mathbf y_i$ and $\mathbf y_j$ respectively.

\begin{lem} \label{lem:auxlemratio}
We have
\[
\frac{\dist(k_{i,j},k_{j,i})}{\dist(g_{i,j},g_{j,i})}=\frac{\|\mathbf s_i-\mathbf s_j\|}{\|\mathbf y_i-\mathbf y_j\|}=\frac{2\lambda_i\lambda_j}{\lambda_iu_j+\lambda_ju_i}.
\]
\end{lem}
\begin{proof}
Denote by $\mathbf c$ the point of intersection $r_0$ with $r$, where, if $r_0$ and $r$ are parallel, then we consider $\mathbf c$ as the corresponding point at infinity. Let $\mathbf c':=\pr(\mathbf c)$. Since $\mathbf c\in h$, the point $\mathbf c'$ is a point at infinity. Without loss of generality we assume that points on the line $r_0$ lie in the following order: $\mathbf z_i, \mathbf x_i,\mathbf x_j,\mathbf z_j$. Denote by $\mathbf w_i$ and $\mathbf w_j$ the orthogonal projections of $\mathbf z_i$ and $\mathbf z_j$ onto the line $r$ respectively. Note that points on the line $r$ lie in the following order: $\mathbf w_i,$ $\mathbf v_i$, $\mathbf v_j$,  $\mathbf w_j$. Moreover, $\mathbf x$ must lie between $\mathbf v_j-\lambda_j\mathbf r$ and $\mathbf v_i+\lambda_i \mathbf r$.

Using the fact that $\mathbf c'$ is a point at infinity, $\{\mathbf s_i, \mathbf y_i, \mathbf y_j, \mathbf s_j, \mathbf c'\}=\pr(\{\mathbf z_i, \mathbf x_i, \mathbf x_j, \mathbf z_j, \mathbf c\})$ and $\pr_0(\{\mathbf z_i, \mathbf x_i, \mathbf x_j, \mathbf z_j, \mathbf c\})=\{\mathbf w_i,\mathbf v_i, \mathbf v_j, \mathbf w_j, \mathbf c\}$, where $\pr_0:r_0\to r$ is the orthogonal projection onto the line $r$, we easily get 
\[
\frac{\|\mathbf s_i-\mathbf s_j\|}{\|\mathbf y_i-\mathbf y_j\|}=(\mathbf s_i, \mathbf y_i; \mathbf s_j, \mathbf c')\cdot (\mathbf s_j,\mathbf y_j;\mathbf y_i,\mathbf c')=(\mathbf z_i, \mathbf x_i; \mathbf z_j, \mathbf c)\cdot (\mathbf z_j, \mathbf x_j; \mathbf x_i, \mathbf c)=
\]
\begin{equation}
\label{equation:fraction1}
=(\mathbf w_i, \mathbf v_i; \mathbf w_j, \mathbf c)\cdot (\mathbf w_j, \mathbf v_j; \mathbf v_i, \mathbf c)=
\frac{\|\mathbf w_i-\mathbf w_j\|}{\|\mathbf v_i-\mathbf v_j\|}\cdot \frac{\|\mathbf v_i-\mathbf c\|}{\|\mathbf w_j-\mathbf c\|} \cdot \frac{\|\mathbf v_j-\mathbf c\|}{\|\mathbf w_i-\mathbf c\|}.
\end{equation}
If $\mathbf c$ is not a point at infinity then using $\Delta(\mathbf c, \mathbf v_i, \mathbf x_i)\sim \Delta(\mathbf c, \mathbf w_i, \mathbf z_i)$ and $\Delta(\mathbf c, \mathbf v_j, \mathbf x_j)\sim \Delta(\mathbf c, \mathbf w_j, \mathbf z_j)$, we have
\begin{equation*}
\label{equation:fractions1}
\frac{\|\mathbf v_i-\mathbf c\|}{\|\mathbf w_i-\mathbf c\|}=\frac{\|\mathbf v_i-\mathbf x_i\|}{\|\mathbf w_i-\mathbf z_i\|}=\frac{\lambda_i}{\|\mathbf w_i-\mathbf z_i\|} \text{ and }\frac{\|\mathbf v_j-\mathbf c\|}{\|\mathbf w_j-\mathbf c\|}=\frac{\|\mathbf v_j-\mathbf x_j\|}{\|\mathbf w_j-\mathbf z_j\|}=\frac{\lambda_j}{\|\mathbf w_j-\mathbf z_j\|}.
\end{equation*}
Note that if $\mathbf c$ is a point at infinity then these equalities are obvious. Substituting the last equality into (\ref{equation:fraction1}), we get
\begin{equation}
\label{equation:fraction2}
\frac{\|\mathbf s_i-\mathbf s_j\|}{\|\mathbf y_i-\mathbf y_j\|}=\frac{\|\mathbf w_i-\mathbf w_j\|}{\|\mathbf v_j-\mathbf v_i\|}\cdot \frac{\lambda_i}{\|\mathbf w_i-\mathbf z_i\|}\cdot \frac{\lambda_j}{\|\mathbf w_j-\mathbf z_j\|}.
\end{equation}
Since $\Delta(\mathbf w_i,\mathbf  z_i, \mathbf x)\sim \Delta(\mathbf v_i,\mathbf  x_i, \mathbf v_i+\lambda_i \mathbf r)$ and $\|\mathbf v_i-\mathbf x_i\|=\lambda_i$, we get
\begin{equation}\label{equation:vverh}\frac{\|\mathbf w_i-\mathbf x\|}{\|\mathbf w_i-\mathbf z_i\|}=\frac{\|\mathbf v_i-\mathbf v_i-\lambda_i \mathbf r\|}{\|\mathbf v_i-\mathbf x_i\|}
\Leftrightarrow
\frac{\|\mathbf w_i-\mathbf x\|}{\|\mathbf r\|}=\|\mathbf w_i-\mathbf z_i\|.\end{equation}
By a similar argument, we obtain
\begin{equation}
\label{equation:vverh2}\frac{\|\mathbf x-\mathbf w_j\|}{\|\mathbf r\|}=\|\mathbf w_j-\mathbf z_j\|.
\end{equation}
From (\ref{equation:vverh}), (\ref{equation:vverh2}) and (\ref{equation:uiuj}) we conclude that
\begin{equation*}
\label{equation:fractions2}
\frac{\|\mathbf w_i-\mathbf w_j\|}{\|\mathbf v_i-\mathbf v_j\|}=\frac{\|\mathbf w_i-\mathbf z_i\|+\|\mathbf w_j-\mathbf z_j\|}{u_i+u_j}.
\end{equation*}
Substituting the last equality into (\ref{equation:fraction2}), we have
\begin{equation}
\label{equation:strahh}
\frac{\|\mathbf s_i-\mathbf s_j\|}{\|\mathbf y_i-\mathbf y_j\|}=\frac{\|\mathbf w_i-\mathbf z_i\|+\|\mathbf w_j-\mathbf z_j\|}{u_i+u_j}\cdot \frac{\lambda_i}{\|\mathbf w_i-\mathbf z_i\|}\cdot \frac{\lambda_j}{\|\mathbf w_j-\mathbf z_j\|}.
\end{equation}
Now we are ready to apply twice the following simple fact.
\begin{lem}\label{lem:ontrapezoid}
Suppose that $\mathbf a_i$ and $\mathbf b_i$ for $1\leq i\leq 3$ are points in $\mathbb R^d$ such that  
$\theta_1 (\mathbf a_1-\mathbf a_2)=\theta_2(\mathbf a_2-\mathbf a_3)$ and $\theta_1(\mathbf b_1-\mathbf b_2)=\theta_2(\mathbf b_2-\mathbf b_3)$, where $\theta_1$ and $\theta_2$ are real numbers. Then
\[
\mathbf b_2-\mathbf a_2=\frac{\theta_1}{\theta_1+\theta_2} (\mathbf b_1-\mathbf a_1)+\frac{\theta_2}{\theta_1+\theta_2}(\mathbf b_3-\mathbf a_3).
\]
\end{lem}
\begin{proof} A simple exercise.
\end{proof}
Denote by $\mathbf x'$ the point of intersection of $\ell(\mathbf x,\ell(\mathbf v_i, \mathbf x_i))$ with $r_0$ (see Figure~\ref{figure:2}). Using Lemma~\ref{lem:ontrapezoid} for $\mathbf w_i, \mathbf z_i, \mathbf x, \mathbf x', \mathbf w_j$ and $\mathbf z_j$,
we obtain 
\begin{equation}
\label{equation:kuku1}
\|\mathbf x-\mathbf x'\|=\frac{2\|\mathbf w_i-\mathbf z_i\|\|\mathbf w_j-\mathbf z_j\|}{\|\mathbf w_i-\mathbf z_i\|+\|\mathbf w_i-\mathbf z_j\|}.
\end{equation}
Using Lemma~\ref{lem:ontrapezoid} for $\mathbf v_i,\mathbf x_i, \mathbf x, \mathbf x', \mathbf v_j$ and $\mathbf x_j$, we have
\begin{equation}
\label{equation:kuku2}
\|\mathbf x-\mathbf x'\|=\frac{u_j}{u_i+u_j}\lambda_i+\frac{u_i}{u_i+u_j}\lambda_j=\frac{\lambda_i u_j+\lambda_j u_i}{u_i+u_j}.
\end{equation}
The comparison of (\ref{equation:kuku1}) and (\ref{equation:kuku2}) shows that
\begin{equation*}\label{equation:strah}
\frac{\|\mathbf w_i-\mathbf z_i\|+\|\mathbf w_j-\mathbf z_j\|}{\|\mathbf w_i-\mathbf z_i\|\|\mathbf w_j-\mathbf z_j\|}=\frac{2}{\|\mathbf x-\mathbf x'\|}=2\frac{u_i+u_j}{\lambda_i u_j+\lambda_j u_i}.
\end{equation*}
Substituting the last equality into (\ref{equation:strahh}), we get
\[
\frac{\|\mathbf s_i-\mathbf s_j\|}{\|\mathbf y_i-\mathbf y_j\|}=\frac{2\lambda_i\lambda_j}{\lambda_i u_j+\lambda_j u_i}.
\]
Lemma~\ref{lem:auxlemratio} is proved.
\end{proof}
\begin{lem}\label{lem:leq2}
If $t_i\cap t_j\subset [\mathbf v_i, \mathbf v_j]$ then 
\[
\frac{2\lambda_i\lambda_j}{\lambda_iu_j+\lambda_ju_i}\leq 2.
\]
\end{lem}
\begin{proof}
Without loss of generality we assume that $\lambda_i\geq \lambda_j$. Note that by definition $u_i, u_j$ are such numbers that $\mathbf x-\mathbf v_i=u_i \mathbf r$ and $\mathbf v_j-\mathbf x=u_j\mathbf r$. Thus if $x\in t_i\cap t_j\subset [\mathbf v_i, \mathbf v_j]$ then $u_i, u_j\geq 0$ and $u_i+u_j\geq\lambda_i\geq \lambda_j$, i.e. $\lambda_i\lambda_j\leq (u_i+u_j)\lambda_j\leq (\lambda_j u_i+\lambda_iu_j)$. The last inequality proves the statement of Lemma \ref{lem:leq2}.
\end{proof}
\section{Proofs of theorems}
\subsection{Proof of Theorem \ref{thm:symmetric}}\label{sec:proofthm1}
Using the notations of Section \ref{sec:auxilirylemmas}, we consider $X_1\subset h_1$, where $h_1$ is a $(d+1)$-dimensional plane. Moreover, by Lemmas \ref{lem:slab} and \ref{lem:auxlemratio} for any $i\ne j\in [n]$ there exist two parallel $d$-dimensional planes $k_{i,j}$ and $k_{j,i}$ such that $\mathbf y_k\in \conv(k_{i,j}\cup k_{j,i})$ for any $k\in [n]$ and 
\begin{equation}\label{eq:bigratio}\frac{\dist(k_{i,j},k_{j,i})}{\dist(g_{i,j}, g_{j,i})}=\frac{2\lambda_i\lambda_j}{\lambda_iu_{j}+\lambda_ju_{i}}.\end{equation}
Since these homothets form a Minkowski arrangement, we have $t_i\cap t_j\subset [\mathbf v_i, \mathbf v_j]$, i.e. by Lemma \ref{lem:leq2} we have that \eqref{eq:bigratio} is less than or equal to~$2$. Therefore, $X_1$ satisfies conditions of Lemma \ref{lem:general} with $\lambda=2$. Thus $n\leq 3^{d+1}$.

\subsection{Proof of Theorem \ref{thm:basic}}\label{sec:proofthm2} 
Consider the following family of pairwise intersecting homothets $\{\mathbf v_i+\lambda_i K:i\in [n] \}$, where $\lambda_{n}:=\lambda_{n-1}$. Without loss of generality assume that $\max_{i\in [n]} \lambda_i=1$. Let us divide the set $[n]$ into $d$ subsets. For any $l\in[d]$ we consider
\[
J_l=\{i\in [n]: \lambda_i\in \mu^{l-1} I \}, \text{ where } \mu=2^{-1/(d-1)}<1, \text{ i.e. } \mu^{d}=\mu/2, \text{ and }
\]
\[
I=I_1\cup I_2\cup I_3\cup\dots:=(\mu,1]\cup (\mu^{d+1},\mu^{d}]\cup (\mu^{2d+1},\mu^{2d}]\cup \cdots.
\]
Obviously, the $J_l$s are not pairwise intersecting sets and their union is $[n]$. We claim that
\begin{equation}
\label{eq:thmaim} |J_l|\leq \left(1+\frac{2}{2-\mu^{-1}}\right)^{d+1}
\end{equation}
Clearly, (\ref{eq:thmaim}) implies the statement of Theorem \ref{thm:basic}: 
\[
n\leq d\left(1+\frac{2}{2-\mu^{-1}}\right)^{d+1}.
\]

It is enough to prove \eqref{eq:thmaim} for $l=1$. Consider the set of homothets $\{\mathbf v_k+\lambda_kK:k\in J_1\}$. Using the notations of Section \ref{sec:auxilirylemmas}, we have that for any $i\ne j$ there exist two parallel $d$-dimensional planes $k_{i,j}$ and $k_{j,j}$ in the $(d+1)$-dimensional plane $h_1$ such that $\mathbf y_k\in \conv(k_{i,j}\cup k_{j,i})$ for any $k\in J_1$ and 
\begin{equation}\label{eq:distance}
\frac{\dist(k_{i,j}, k_{j,i})}{\dist(g_{i,j},g_{j,i})}=\frac{2\lambda_i\lambda_j}{\lambda_iu_j+\lambda_ju_i}.
\end{equation}
By Lemma \ref{lem:general}, it is enough to prove that the right hand side of \eqref{eq:distance} is at most $\frac{2}{2-\mu^{-1}}>2$. Consider two cases:

1) $i, j\in I_k$ for some $k$. Assume that $i<j$ thus $\mathbf v_j-\mathbf v_i=\lambda_i \mathbf r$. If $\lambda_i\geq \lambda_j$ then we have $t_i\cap t_j\subset [\mathbf v_i, \mathbf v_j]$, i.e. by Lemma \ref{lem:leq2}, we have that \eqref{eq:distance} is at most 2. Assume that $\lambda_j>\lambda_i$. Since $\mathbf x\in [\mathbf v_j-\lambda_j \mathbf r, \mathbf v_i+\lambda_i\mathbf r]=[\mathbf v_j-\lambda_j\mathbf r,\mathbf v_j]$, we have $u_i+u_j=\lambda_i$, $0\leq u_j\leq \lambda_j$. Therefore, using $\lambda_j/\lambda_i<\mu^{-1}$ (because $i,j\in I_k$) we have 
\[
\frac{\lambda_i u_j+\lambda_j u_i}{\lambda_i\lambda_j}=u_j\left(\frac{1}{\lambda_j}-\frac{1}{\lambda_i}\right)+\frac{u_i+u_j}{\lambda_i}\geqslant \lambda_j  \left(\frac{1}{\lambda_j}-\frac{1}{\lambda_i}\right) + 1> 2-\mu^{-1},
\]
i.e. the right hand side of \eqref{eq:distance} is at most $\frac{2}{2-\mu^{-1}}$.

2) $i\in I_k, j\in I_l$ for some $k<l$. Note that $\lambda_i>2\lambda_j$ (see the definition of $I_m$) thus it is impossible that $\mathbf v_j-\mathbf v_i=\lambda_j\mathbf r$. Indeed, in such case $\mathbf v_i+\lambda_i \partial K$ and $\mathbf v_j+\lambda_j \partial K$ do not intersect each other because of the triangle inequality, a contradiction. Therefore, $\mathbf v_j-\mathbf v_i=\lambda_i \mathbf r$, i.e. $t_i\cap t_j\subset [\mathbf v_i, \mathbf v_j]$, thus \eqref{eq:distance} is at most $2$.

Theorem \ref{thm:basic} is proved.

\subsection{Proof of Theorem \ref{thm:kdistancessets}}\label{sec:proofthm3}
Assume that there exists a $k$-distance set $\{\mathbf x_i:i\in [n]\}$ in the $d$-dimensional Minkowski space with an $o$-symmetric convex body $K$ as the unit ball, where 
\[
n=k^{f(d)},\ f(d)=\left\lfloor d\left(1+\frac{2}{2-2^{1/(d-1)}}\right)^{d+1}\right\rfloor=O(3^dd).
\]
We will construct a set $Y=\{\mathbf y_i:i\in [f(d)+1]\}$ in the same $d$-dimensional Minkowski space such that $\|\mathbf y_i-\mathbf y_j\|_K=\lambda_i$ for any $1\leq i<j\leq f(d)+1$, where $\lambda_i$ are some positive real numbers, using the following algorithm.
\begin{itemize}
\item[0.] Set $A:=[n]$, $Y:=\{\mathbf y_1:=\mathbf x_1\}$, $l:=1$.
\item[1.] Let $\lambda_l$ be a positive real number such that the cardinality of the set $$A':=\{j:\|\mathbf y_l-\mathbf x_j\}\|_K=\lambda_l, j\in A\}$$ is at least $k^{f(d)-l}$ (such $i_l$ exists because $|A|\geq k^{f(d)-l+1}$ and there are $k$ distances occurring between points of $\{\mathbf x_i:i\in A\subseteq[n]\}$). Put $A:=A'$.
\item[2.] Choose any $j\in A$ and put $\mathbf y_{l+1}:=\mathbf x_j$. Add $\mathbf y_{l+1}$ to the set $Y$.
\item[3.] If $l<f(d)$ then $l:=l+1$ and return to Step 1, else, output $Y$, and finish. 
\end{itemize}
Obviously, the existence of the set $Y$ contradicts Theorem \ref{thm:basic}, therefore, we get a contradiction with our assumption that there exists a $k$-distance set consisting of $k^{f(d)}$ points in $\mathbb R^d$.

\section*{Acknowledgment.} We are grateful to M\'arton Nasz\'odi and Konrad Swanepoel for stimulating and fruitful discussions, to anonymous referees for valuable comments that helped to significantly improve the presentation of the paper. We wish to thank one of the referees for bringing to our attention the idea to use cross-ratios in the proof of Lemma~\ref{lem:auxlemratio}. The author was partially supported by ISF grant no. 409/16, and by the Russian Foundation for Basic Research, grants \textnumero~15-31-20403 (mol\_a\_ved), \textnumero~15-01-99563 A, \textnumero~15-01-03530 A.

\end{document}

%% file: pic1.tex
\begin{figure}[h]
\centering
\definecolor{qqffff}{rgb}{0.,1.,1.}
\definecolor{ttffcc}{rgb}{0.2,1.,0.8}
\definecolor{ttffqq}{rgb}{0.2,1.,0.}
\definecolor{ffqqqq}{rgb}{1.,0.,0.}
\begin{tikzpicture}[line cap=round,line join=round,>=triangle 45,x=1.42310839331363cm,y=1.8183246678568954cm]
\clip(0.1,-0.3) rectangle (8.2,2.8);
\draw [thin] (-0.7,0.)--(9,0.);
\draw [dotted] (2.5079766724977963,0.) -- (2.5079766724977963,2.2661343660901636);
\draw [dotted] (6.078163084142602,0.) -- (6.078163084142602,1.8405246709963947);
\draw [dotted] (5.010350331838055,0.) -- (5.010350331838055,1.4201788804318947);
\begin{scriptsize}
\node [above] at (1.5,0.) {$r$};
\draw [fill=black] (0.2418423064076329,0.) circle (1.pt);
\node [below right] at (0.2418423064076329,0.) {$\mathbf v_i-\lambda_i\mathbf r$};
\draw [fill=black] (4.77411103858796,0.) circle (1.pt);
\node [below] at (4.77411103858796,0.) {$\mathbf v_i+\lambda_i\mathbf r$};
\draw [fill=black] (3.5901714514061602,0.) circle (1.pt);
\node [above left] at (3.5901714514061602,0.) {$\mathbf v_j-\lambda_j\mathbf r$};
\draw [fill=black] (6.43052921226995,0.) circle (1.pt);
\node [above right] at (6.43052921226995,0.) {$\mathbf v_j+\lambda_j\mathbf r$};
\draw [fill=black] (4.2376384131462075,0.) circle (1.pt);
\node [below left] at (4.2376384131462075,0.) {$\mathbf v'_k-\lambda_k\mathbf r$};
\draw [fill=black] (7.918687755138998,0.) circle (1.pt);
\node [below left] at (8.049635844650437,0.) {$\mathbf v'_k+\lambda_k\mathbf r$};
\draw [fill=white] (2.5079766724977963,0.) circle (1.0pt);
\node [above left] at (2.5079766724977963,0.) {$\mathbf v_i$};
\draw [fill=white] (6.078163084142602,0.) circle (1.pt);
\node [below] at (6.078163084142602,0.) {$\mathbf v'_k$};
\draw [fill=white] (5.010350331838055,0.) circle (1.0pt);
\node [above right] at (5.010350331838055,0.) {$\mathbf v_j$};
\draw  (4.51609979025661,0.) -- ++(-3.5pt,-3.5pt) -- ++(3.5pt,3.5pt)-- ++(3.5pt,-3.5pt) -- ++(-3.5pt,3.5pt) --++(10,10)-- ++ (-10,-10) -- ++ (-10,10);
\node [above] at (4.51609979025661,0.) {$\mathbf x$};
\draw  [thin, dashed] (4.77411103858796,0.) -- ++ (-10,10);
\draw  [thin, dashed]
(3.5901714514061602,0.) -- ++(10,10);
\draw [fill=black] (2.5079766724977963,2.2661343660901636) circle (1.pt);
\node [above] at (2.5079766724977963,2.2661343660901636) {$\mathbf x_i$};
\draw [fill=black] (5.010350331838055,1.4201788804318947) circle (1.pt);
\node [above] at (5.010350331838055,1.4201788804318947) {$\mathbf x_j$};
\draw [fill=black] (6.078163084142602,1.8405246709963947) circle (1.pt);
\node [above] at ( 6.078163084142602,1.8405246709963947) {$\mathbf x'_k$};
\draw (1.7,2.64) node {$b_i$};
\draw (2.3,2.64) node {$a_i$};
\draw (7.02,2.64) node {$b_j$};
\draw (6.43,2.64) node {$a_j$};
\end{scriptsize}
\end{tikzpicture}
\caption{Plane $\pi$}
\label{figure:1}
\end{figure}

%% file: pic2.tex
\begin{figure}[h]
\centering
\begin{tikzpicture}[line cap=round,line join=round,>=triangle 45,x=2.3144752040715297cm,y=2.662693758530262cm]
\clip(2.0,-0.15) rectangle (8.2,2.55);
\draw [thin] (2.,0.)--(9,0.);
\draw [dotted] (2.5079766724977963,0.) -- (2.5079766724977963,2.2661343660901636);
\draw [dotted] (5.283689307718597,0.) -- (5.283689307718597,1.1468399045513529);
\draw (8.127711254599488,0.) -- (2.075619014630229,2.4404807756263813) ;
\draw [dotted] (2.075619014630229,2.4404807756263813)-- (2.075619014630229,0.);
\draw [dotted] (5.553955847728781,1.037856057472171)-- (5.553955847728781,0.);
\draw [dotted] (4.51609979025661,1.45) -- (4.51609979025661,0.);
\draw [fill=black] (4.51609979025661,1.457) circle (1.pt);
\draw [dashed, color=black] (2.5079766724977963,2.2661343660901636)-- (4.77411103858796,0.);
\begin{scriptsize}
\draw (2.951041993674191,0.06087503127319516) node {$r$};
\draw [fill=black] (4.77411103858796,0.) circle (1.pt);
\node [above right] at (4.77411103858796,0.) {$\mathbf v_i+\lambda_i \mathbf r$};
\draw [fill=black] (4.136849403167244,0.) circle (1.pt);
\node [above left] at (4.136849403167244,0.) {$\mathbf v_j-\lambda_j \mathbf r$};
\draw [dashed] (4.136849403167244,0.) --(5.283689307718597,1.1468399045513529); 
\draw [fill=white] (2.5079766724977963,0.) circle (1.pt);
\node [below] at (2.5079766724977963,0.) {$\mathbf v_i$};
\draw [fill=white] (5.283689307718597,0.) circle (1.pt);
\node [below] at (5.283689307718597,0.) {$\mathbf v_j$};
\draw [color=black] (4.51609979025661,0.)-- (5.553955847728781,1.037856057472171);
\draw (2.075619014630229,2.4404807756263813) --(4.51609979025661,0.);
\node [below] at (4.51609979025661,0.) {$\mathbf x$};
\draw [fill=black] (2.5079766724977963,2.2661343660901636) circle (1.pt);
\node [above] at (2.5079766724977963,2.2661343660901636) {$\mathbf x_i$};
\draw [fill=black] (5.283689307718597,1.1468399045513529) circle (1.pt);
\node at (4,1.75) {$r_0$};
\node [above] at (5.283689307718597,1.1468399045513529) {$\mathbf x_j$};
\draw [fill=black] (2.075619014630229,2.4404807756263813) circle (1.pt);
\node [above] at (2.075619014630229,2.4404807756263813) {$\mathbf z_i$};
\draw [fill=black] (5.553955847728781,1.037856057472171) circle (1.pt);
\node [above] at (5.553955847728781,1.037856057472171) {$\mathbf z_j$};
\draw [fill=black] (2.075619014630229,0.) circle (1.pt);
\node [below] at (2.075619014630229,0.) {$\mathbf w_i$};
\draw [fill=black] (5.553955847728781,0.) circle (1.pt);
\node [below] at (5.553955847728781,0.) {$\mathbf w_j$};
\node [above] at (4.51609979025661,1.457) {$\mathbf x'$};
\draw [fill=black] (8.127711254599488,0.) circle (1.pt);
\node [below] at (8.127711254599488,0.) {$\mathbf c$};
\end{scriptsize}
 \end{tikzpicture}
 \caption{Plane $\pi$}
 \label{figure:2}
\end{figure}

%% file: main.bbl
\begin{thebibliography}{100}
\bibitem{DG62} L.~Danzer, B.~Gr\"unbaum, \emph{\"Uber zwei Probleme bez\"uglich konvexer K\"orper von P. Erd\H os und von V. L. Klee}, Math. Z. {\bf 79} (1962), 95--99.

\bibitem{FL94} Z.~F\"uredi, P.\,A.~Loeb, \emph{On the best constant for the Besicovitch covering theorem}, Proc. Amer. Math. Soc. {\bf121}(4) (1994), 1063--1073.

\bibitem{LN09} Zs.~L\'angi, M.~Nasz\'odi, \emph{On the Bezdek-Pach conjecture for centrally symmetric
convex bodies}, Canad. Math. Bull. {\bf52}(3) (2009), 407--415.

\bibitem{NPS16} M.~Nasz\'odi, J.~Pach, K.~Swanepoel, \emph{Arrangements of homothets of a convex body}, \href{http://arxiv.org/pdf/1608.04639.pdf}{arXiv:1608.04639}, submitted.

\bibitem{N06} M.~Nasz\'odi, \emph{On a conjecture of K\'aroly Bezdek and J\'anos Pach}, Period. Math. Hungar. {\bf53}(1-2) (2006), 227--230.

\bibitem{P71}M.~Petty, \emph{Equilateral sets in Minkowski spaces}, Proc. Amer. Math. Soc. {\bf 29} (1971), 369--374.

\bibitem{FT67} L.\,F.~T\'oth, \emph{Research problem}, Period. Math. Hungar. {\bf31}(2) (1995), 165--166.

\bibitem{S07} K.~Swanepoel, \emph{Cardinalities of $k$-distance sets in Minkowski spaces}, Discrete Mathematics {\bf 197/198} (1999), 759--767.

\bibitem{S16} K.~Swanepoel, \emph{Combinatorial distance geometry in normed spaces}, \href{https://arxiv.org/abs/1702.00066}{arXiv:1702.00066}, submitted.
\end{thebibliography}
